\documentclass[12pt,reqno]{amsart}
\usepackage{amsmath}
\usepackage{amssymb}
\usepackage[left=2.54cm,top=2.54cm,right=2.54cm,bottom=2.54cm]{geometry}
\usepackage{epsfig}
\usepackage{color}                    

\begin{document}
\newtheorem{theorem}{Theorem}[section]
\newtheorem{lemma}[theorem]{Lemma}
\newtheorem{claim}[theorem]{Claim}
\newtheorem{definition}[theorem]{Definition}
\newtheorem{conjecture}[theorem]{Conjecture}
\newtheorem{proposition}[theorem]{Proposition}
\newtheorem{algorithm}[theorem]{Algorithm}
\newtheorem{corollary}[theorem]{Corollary}
\newtheorem{observation}[theorem]{Observation}
\newtheorem{problem}[theorem]{Open Problem}
\newcommand{\noin}{\noindent}
\newcommand{\ind}{\indent}
\newcommand{\al}{\alpha}
\newcommand{\om}{\omega}
\newcommand{\R}{{\mathbb R}}
\newcommand{\N}{{\mathbb N}}
\newcommand\eps{\varepsilon}
\newcommand{\E}{\mathbb E}
\newcommand{\Prob}{\mathbb{P}}
\newcommand{\pl}{\textrm{C}}
\newcommand{\dang}{\textrm{dang}}
\renewcommand{\labelenumi}{(\roman{enumi})}
\newcommand{\bc}{\bar c}
\newcommand{\G}{{\mathcal{G}}}
\newcommand{\expect}[1]{\E \left [ #1 \right ]}
\newcommand{\floor}[1]{\left \lfloor #1 \right \rfloor}
\newcommand{\ceil}[1]{\left \lceil #1 \right \rceil}
\newcommand{\of}[1]{\left( #1 \right)}
\newcommand{\set}[1]{\left\{ #1 \right\}}
\newcommand{\angs}[1]{\left\langle #1 \right\rangle}
\newcommand{\sqbs}[1]{\left[ #1 \right]}
\newcommand{\sm}{\setminus}
\newcommand{\bfrac}[2]{\of{\frac{#1}{#2}}}
\renewcommand{\k}{\kappa}
\renewcommand{\l}{\ell}
\renewcommand{\b}{\beta}
\renewcommand{\a}{\alpha}
\renewcommand{\o}{\omega}
\newcommand{\Bin}{\textrm{Bin}}
\newcommand{\Cc}{{\mathcal{C}}}
\newcommand{\Nn}{{\mathcal{N}}}
\newcommand{\Rr}{{\mathcal{R}}}
\newcommand{\Ee}{{\mathcal{E}}}
\newcommand{\jlo}{2j\frac{\log (n/j)}{\log j}}
\newcommand{\opoo}{(1+o(1))}
\renewcommand{\i}{\ell}

\title[{Power of $k$ choices and rainbow spanning trees in random graphs}]{Power of $k$ choices \\ and \\ rainbow spanning trees in random graphs}

\author{Deepak Bal}
\address{Department of Mathematics, Miami University, Oxford, OH, 45056, U.S.A.}
\email{baldc@miamioh.edu}

\author{Patrick Bennett}
\address{Department of Computer Science, University of Toronto, Toronto, ON, Canada, M5S 3G4}
\email{patrickb@cs.toronto.edu}

\author{Alan Frieze}
\address{Department of Mathematical Sciences, Carnegie Mellon University, 5000 Forbes Av., 15213, Pittsburgh, PA, U.S.A}
\thanks{The third author is supported in part by NSF Grant CCF0502793}
\email{alan@random.math.cmu.edu}

\author{Pawe\l{} Pra\l{}at}
\address{Department of Mathematics, Ryerson University, Toronto, ON, Canada, M5B 2K3}
\thanks{The fourth author is supported in part by NSERC and Ryerson University}
\email{\texttt{pralat@ryerson.ca}}

\maketitle
\begin{abstract}
We consider the Erd\H{o}s-R\'enyi random graph process, which is a stochastic process that starts with $n$ vertices and no edges, and at each step adds one new edge chosen uniformly at random from the set of missing edges. Let $\G(n,m)$ be a graph with $m$ edges obtained after $m$ steps of this process. Each edge $e_i$ ($i=1,2,\ldots, m$) of $\G(n,m)$ independently chooses precisely $k \in \N$ colours, uniformly at random, from a given set of $n-1$ colours (one may view $e_i$ as a multi-edge). We stop the process prematurely at time $M$ when the following two events hold: $\G(n,M)$ is connected and every colour occurs at least once ($M={n \choose 2}$ if some colour does not occur before all edges are present; however, this does not happen asymptotically almost surely). The question addressed in this paper is whether $\G(n,M)$ has a rainbow spanning tree (that is, multicoloured tree on $n$ vertices). Clearly, both properties are necessary for the desired tree to exist. 

In 1994, Frieze and McKay investigated the case $k=1$ and the answer to this question is ``yes'' (asymptotically almost surely). However, since the sharp threshold for connectivity is $\frac {n}{2} \log n$ and the sharp threshold for seeing all the colours is $\frac{n}{k} \log n$, the case $k=2$ is of special importance as in this case the two processes keep up with one another. In this paper, we show that asymptotically almost surely the answer is ``yes'' also for $k \ge 2$.
\end{abstract}

\section{Introduction and the main result}

In this paper, we consider the \textbf{Erd\H{o}s-R\'enyi random graph process}, which is a stochastic process that starts with $n$ vertices and no edges, and at each step adds one new edge chosen uniformly at random from the set of missing edges. Formally, let $N={n \choose 2}$ and let $e_1, e_2, \ldots, e_N$ be a random permutation of the edges of the complete graph $K_n$. The graph process consists of the sequence of random graphs $( \G(n,m) )_{m=0}^{N}$, where $\G(n,m)=(V,E_m)$, $V = [n] := \{1, 2, \ldots, n\}$, and $E_m = \{ e_1, e_2, \ldots, e_m \}$. It is clear that $\G(n,m)$ is a graph taken uniformly at random from the set of all graphs on $n$ vertices and $m$ edges. (See, for example,~\cite{bol, JLR} for more details.)

All asymptotics throughout are as $n \rightarrow \infty $ (we emphasize that the notations $o(\cdot)$ and $O(\cdot)$ refer to functions of $n$, not necessarily positive, whose growth is bounded). We say that an event in a probability space holds \textbf{asymptotically almost surely} (or \textbf{a.a.s.}) if the probability that it holds tends to $1$ as $n$ goes to infinity. We often write $\G(n,m)$ when we mean a graph drawn from the distribution $\G(n,m)$.  

\bigskip

A set of edges $S$ is said to be \textbf{rainbow} if each edge of $S$ is in a different colour. When considering adversarial (worst-case) colouring, the guaranteed existence of a rainbow structure is called an \emph{Anti-Ramsey}\/ property.  Erd\H{o}s, Ne\v{s}et\v{r}il, and R\"odl \cite{ENR}, Hahn
and Thomassen~\cite{HT} and Albert, Frieze, and Reed~\cite{AFR} (correction in Rue~\cite{Rue}) considered colourings of the edges of the complete graph $K_n$ where no colour is used more than $k$ times. It was shown in~\cite{AFR} that if $k \leq n/64$, then there must be a rainbow Hamilton cycle.  Cooper and Frieze~\cite{CF1} proved a random graph threshold for this property to hold in almost every graph in the space studied.

Let us now focus on the random colouring situation. Cooper and Frieze~\cite{CF} showed that if $m\geq Kn\log n$ and there are at least $Kn$ colours available, for $K$ sufficiently large, then a.a.s. $\G(n,m)$ contains a rainbow Hamilton cycle. This was improved by Frieze and Loh~\cite{FL} to give $K=1+o(1)$.

In this paper we are concerned with the existence of rainbow spanning trees of $\G(n,m)$. Suppose that each edge $e_i$ ($i=1,2,\ldots, m$) of $\G(n,m)$ independently chooses precisely $k \in \N$ colours, uniformly at random, from a given set $W$ of $n-1$ colours. In other words, each edge $e_i$ has assigned a set $c(e_i)$ of $k$ colours; for every $i$ and every $S \subseteq W, |S|=k$, we have $c(e_i) = S$ with probability ${n-1 \choose k}^{-1}$. (It is convenient to view $e_i$ as a multi-edge, hence $\G(n,m)$ may be viewed as a multi-graph on $km$ coloured edges.)

\bigskip

We are concerned with the following three events:
\begin{eqnarray*}
\Cc_m &=& \{ \text{$\G(n,m)$ is connected} \}, \\
\Nn_m = \Nn_m(k) &=& \{ \text{$\G(n,m)$ contains edges in every colour} \},\\
\Rr_m = \Rr_m(k) &=& \{ \text{$\G(n,m)$ has a rainbow spanning tree} \}.
\end{eqnarray*}
Let $\Ee_m$ stand for one of the above three sequences of events and let 
$$
m_{\Ee} = \min\{m \in \N : \Ee_m \text{ occurs}\},
$$
provided that such an $m$ exists. (Note that $m_{\Cc}$ is always defined but the other two might not be.) Moreover, if $m_{\Rr}$ is defined, then so is $m_{\Nn}$ and clearly
$$
m_{\Rr} \ge \max\{m_{\Cc}, m_{\Nn} \}.
$$ 

In 1994, Frieze and McKay~\cite{Frieze_Mckay} investigated the case $k=1$ and they showed that a.a.s.\ $m_{\Rr} = \max\{m_{\Cc}, m_{\Nn} \}$. It is well known that the sharp threshold for connectivity is $\frac {n}{2} \log n$; in fact,
\begin{equation}\label{eq:thr_conn}
\Prob ( \Cc_m ) = (1+o(1)) e^{-e^{-c}},
\end{equation}
provided that $m = \frac {n}{2} (\log n + c), c \in \R$ (as usual, see, for example~\cite{bol, JLR}). Moreover, it follows from the coupon collector problem that the sharp threshold for seeing all the colours is $n \log n$ (for $k=1$). Therefore, processes corresponding to the two obvious necessary conditions are not synchronized for $k=1$. On the other hand, it can be easily generalized and can be shown that for any $k \in \N$ we have
\begin{equation}\label{eq:thr_colours}
\Prob (\Nn_m) = (1+o(1)) e^{-e^{-c}},
\end{equation}
provided that $m = \frac {n}{k} (\log n + c), c \in \R$ (see, Lemma~\ref{lem:coupon-k}). Hence, $k=2$ is of special importance as in this case the two processes keep up with one another. 

\bigskip

In this paper, we generalize the result of Frieze and McKay~\cite{Frieze_Mckay} and show the following result. 
\begin{theorem}\label{thm:main}
For every $k \ge 2$, we have that a.a.s.
$$
m_{\Rr} = \max\{m_{\Cc}, m_{\Nn} \}.
$$ 
\end{theorem}
Since events $\Cc_m$ and $\Nn_m$ are independent, the following corollary follows immediately from~(\ref{eq:thr_conn}) and~(\ref{eq:thr_colours}).
\begin{corollary}
Let $k \ge 2$ and let $m = m(k,n) = \frac {n}{k} (\log n + c)$ for some $c \in \R$. Then,
$$
\Prob (\Rr_m) = 
\begin{cases}
(1+o(1)) e^{-2e^{-c}}, & \text { if } k=2 \\
(1+o(1)) e^{-e^{-c}}, & \text { if } k \ge 3.
\end{cases}
$$
\end{corollary}

The whole paper is devoted to prove Theorem~\ref{thm:main}. In Section~\ref{sec:condition}, we introduce an alternative, very convenient, way of checking whether the coloured graph has the desired property. A sufficient and necessary condition for the existence of a rainbow spanning tree is introduced that uses the result of Edmonds on the matroid intersection problem. In Section~\ref{sec:proof}, we show that a.a.s.\ the condition holds at time $\max\{m_{\Cc}, m_{\Nn} \}$.

\section{Sufficient and necessary condition for the existence of a rainbow spanning tree}\label{sec:condition}

A finite \textbf{matroid} $M$ is a pair $(E,\mathcal{I})$, where $E$ is a finite set (called the \textbf{ground set}) and $\mathcal{I}$ is a family of subsets of $E$ (called the \textbf{independent sets}) with the following properties:
\begin{itemize}
\item $\emptyset \in \mathcal{I}$,
\item for each $A' \subseteq A \subseteq E$, if $A \in \mathcal{I}$, then $A' \in \mathcal{I}$ (\textbf{hereditary property}),
\item if $A$ and $B$ are two independent sets of $\mathcal{I}$ and $A$ has more elements than $B$, then there exists an element in $A$ that when added to $B$ gives a larger independent set than $B$ (\textbf{augmentation property}).
\end{itemize}
A maximal independent set (that is, an independent set which becomes dependent on adding any element of E) is called a \textbf{basis} for the matroid. An observation, directly analogous to the one of bases in linear algebra, is that any two bases of a matroid $M$ have the same number of elements. This number is called the \textbf{rank} of $M$. For more on matroids see, for example,~\cite{Oxley}.

\smallskip

In order to investigate the existence of a rainbow spanning tree we are going use the result of Edmonds on the matroid intersection problem~\cite{Edmonds}. Suppose that $M_1$ and $M_2$ are two matroids over a common ground set $E$ with rank functions $r_1$ and $r_2$ respectively. Edmonds' general theorem shows that
\begin{equation}\label{eq:Edmonds}
\max \big( |I| : I \text{ is independent in both matroids}\big) = \min_{\substack{E_1 \cup E_2 = E \\  E_1 \cap E_2 = \emptyset}} \big( r_1 (E_1) + r_2 (E_2) \big),
\end{equation}
where $r_i(E_i)$ is the rank of the matroid induced by $E_i$. In our application, the common ground set $E$ is the set of coloured multi-edges of $\G(n,m)$. $M_1$ is the cycle matroid; that is, $S \subseteq E$ is independent in $M_1$ if $S$ induces a graph with no cycle (colours are ignored, two parallel edges are considered to be a cycle of length 2). Hence, for every $S \subseteq E$ we have $r_1(S) = n - \kappa(S)$, where $\kappa(S)$ is the number of components of the graph $G=(V,S)$ induced by $S$. $M_2$ is the partition matroid associated with the colours; that is, $S \subseteq E$ is independent in $M_2$ if $S$ has no two edges in the same colour. This time, for every $S \subseteq E$ we have that $r_2(S)$ is the number of distinct colours occurring in $S$. We get immediately the following useful lemma.
\begin{lemma}\label{lem:condition}
Let $G$ be a multigraph on $n$ vertices in which each edge is coloured with a colour from a set $W$ of cardinality $n-1$. A necessary and sufficient condition for the existence of a rainbow spanning tree in $G$ is that for every $I \subseteq W$ we have
\begin{equation}\label{eq:condition}
\kappa(G_I) \le n - |I|,
\end{equation}
where $G_I$ is the graph induced by the set of edges coloured with a colour from $I$.
\end{lemma}
\begin{proof}
Clearly, $G$ has a rainbow spanning tree if and only if $G$ contains a set $S$ of coloured edges of size $n-1$ such that $S$ is independent both in $M_1$ ($S$ induces a spanning tree) and in $M_2$ ($S$ is rainbow). Since no set of size at least $n$ is independent (in any matroid), the necessary and sufficient condition is that the right side of~(\ref{eq:Edmonds}) is at least $n-1$. Hence, the desired condition is that for every partition of the edge set $E$ into $E_1$ and $E_2$ we have $r_1 (E_1) + r_2 (E_2) \ge n-1$.

Let us fix a partition into $E_1$ and $E_2$. Let $J$ be the set of colours occurring in $E_2$, $E'_2$ be the set of edges coloured with a colour from $J$, and $E'_1 = E \setminus E'_2$. Clearly, $(E'_1, E'_2)$ is also a partition of $E$, $E_2 \subseteq E'_2$ and so $E'_1 \subseteq E_1$, and $r_2(E_2) = r_2(E'_2) = |J|$. Moreover, since $E'_1 \subseteq E_1$, $r_1(E'_1) \le r_1(E_1)$ and so
$$
r_1 (E_1) + r_2 (E_2) \ge r_1 (E'_1) + r_2 (E'_2).
$$
Therefore, without loss of generality, we may restrict ourselves to sets $E_2$ containing all edges of colour from some set $J$ and then take $I = W \setminus J$. The condition to verify is the following:
$$
n-1 \le r_1 (E_1) + r_2 (E_2) = (n - \kappa(G_I)) + (n-1-|I|)
$$
which is equivalent to~(\ref{eq:condition}) and the proof is finished. 
\end{proof}

\section{Proof of Theorem~\ref{thm:main}}\label{sec:proof}

Before we move to the proof of the main result, we will show that~(\ref{eq:thr_colours}) holds. 

\begin{lemma}\label{lem:coupon-k}
For any $k \in \N$,
$$
\Prob (\Nn_m) = (1+o(1)) e^{-e^{-c}},
$$
provided that $m = \frac {n}{k} (\log n + c), c \in \R$. 
\end{lemma}
\begin{proof}
Consider $\G(n,m)$ and for each $i \in W$, let $Q_i$ be the event that colour $i$ does not occur in $\G(n,m)$. Let $X_i$ be the indicator random variable for $Q_i$, and let $X = \sum_{i \in W} X_i$. Clearly, $\Nn_m$ occurs if and only if $X=0$. 

It is straightforward to see that for each $i \in W$,
\begin{eqnarray*}
\Prob(Q_i) &=& \left( \frac { {n-2 \choose k} }{ {n-1 \choose k} } \right)^m = \left( 1 - \frac {k}{n-1} \right)^m = \exp \left( - \frac {km}{n}  \left( 1 + O( n^{-1} ) \right) \right) \\
&=& (1+o(1)) \frac {e^{-c}}{n}.
\end{eqnarray*}
Hence, $\E [ X ] = (1+o(1)) e^{-c}$. Similarly, one can show that for every fixed $r \in \N$,
$$
\E \left[ {X \choose r} \right] = (1+o(1)) \frac { (e^{-c})^r }{r!}.
$$
It follows from the Brun's sieve (see, for example, Section 8.3 in~\cite{AS}) that $\Prob(X=0) = (1+o(1)) e^{-e^{-c}}$ and the result holds.
\end{proof}

\bigskip

Fix $k \in \N \setminus \{1\}$ and let $\omega = \omega(n)$ be any function tending to infinity together with $n$. Let
$$
m_- := \left\lfloor \frac {n}{2} ( \log n - \omega ) \right \rfloor \quad \text{ and } m_+ := \left\lceil \frac {n}{2} ( \log n + \omega ) \right \rceil.
$$
(Here we make sure $m_-$ and $m_+$ are both integers but later on expressions such as $i=n-\b n / \log n$   that clearly have to be an integer, we round up or down but do not specify which: the choice of which does not affect the argument.)
It is known that a.a.s.\ $\G(n,m_-)$ is not connected and so a.a.s.\ $m_{\Rr} \ge m_{\Cc} \ge m_-$. On the other hand, a.a.s.\ $\G(n,m_+)$ is connected and so a.a.s.\ $m_{\Cc} \le m_+$. Moreover, it follows from Lemma~\ref{lem:coupon-k} that a.a.s.\ also $m_{\Nn} \le m_+$. We get that a.a.s.
\begin{equation}\label{eq:range_for_m}
m_- \le \max\{m_{\Cc}, m_{\Nn} \} \le m_+.
\end{equation}

Fix $m$ such that $m_- \le m \le m_+$ and for a given $i$, let us define the following events:
\begin{eqnarray*}
A_i &=& \{ \exists I \subseteq W, |I|=i : \kappa(G_I) \ge n - |I| + 1 \}, \\
B_i &=& \{ \exists I \subseteq W, |I|=i : \kappa(G_I) \ge n - |I| + 1 \text{ and all colours from $I$ are present in $G_I$}\}, \\
C_i &=& \{ \exists I \subseteq W, |I|=i : \kappa(G_I) \ge n - |I| + 1 \text{ and } \forall I \subseteq W, |I|<i : \kappa(G_I) \le n - |I|\}, \\
D_i &=& \{ \exists I \subseteq W, |I|=i : \kappa(G_I) \ge n - |I| + 1 \text{ and $G_W$ is connected}\}, 
\end{eqnarray*}
where $G_I$ is the subgraph of $\G(n,m)$ induced by the set of edges coloured with a colour from $I$. Suppose that $m \ge \max\{ m_{\Cc}, m_{\Nn} \}$ and there is no rainbow spanning tree. Note that $A_1$ cannot occur since all colours are present in $\G(n,m)$ and so $\kappa(G_I) \le n-1$ for all $I$ such that $|I|=1$ ($G_I$ has at least one edge). Note also that for every $1 \le j \le k-1$, $A_{n-1-j}$ cannot occur since for every $J \subseteq W$ of size at most $k-1$, $\kappa(G_{W \setminus J}) = 1$ (observe that $G_{W \setminus J} = \G(n,m)$ since one needs to remove at least $k$ colours for two vertices that are connected in $G_W$ to become disconnected, and so $G_{W \setminus J}$ is connected). Hence, $A_i$ has to occur for some $i \in [2, n-1-k]$; otherwise, Lemma~\ref{lem:condition} would imply that $\G(n,m)$ has a rainbow spanning tree. But if $A_i$ holds for some $i \in [2, n- 1-k]$, then $B_i$ holds too (since all colours are present at time $m \ge \max\{ m_{\Cc}, m_{\Nn} \} \ge m_{\Nn}$), $C_j$ holds for some $j \in [2, i]$ (since $A_1$ cannot occur), and $D_i$ holds as well (since $G_W$ is connected at time $m \ge \max\{ m_{\Cc}, m_{\Nn} \} \ge m_{\Cc}$).

Our goal is to use the following estimation and show that the right hand side is $o(1)$:
\begin{align}
\Prob &\left( m_{\Rr} > \max\{ m_{\Cc}, m_{\Nn} \} \right) \label{eq:sum} \\
&\le o(1) + \sum_{m=m_-}^{m_+} \left( \Prob_m(B_2) + \sum_{i=3}^{n/(\b \sqrt{\log n})} \Prob_m(C_i) + \sum_{i=n/(\b \sqrt{\log n})}^{n- \b n / \log n} \Prob_m(A_i) +  \sum_{i=n-\b n / \log n}^{n-1-k} \Prob_m(D_i) \right), \nonumber
\end{align}
where $\Prob_m$ indicates that the corresponding probability refers to the $\G(n,m)$ model, the $o(1)$ term is the probability that $m_{\Rr} \notin [m_-, m_+]$ (see~(\ref{eq:range_for_m})), and $\b$ is a sufficiently large constant that will be determined later. 

\bigskip

Our results refer to the random graph process. However, it will be sometimes easier to work with the $G(n,p)$ model instead of $\G(n,m)$. The \textbf{random graph} $G(n,p)$ consists of the probability space $(\Omega, \mathcal{F}, \Prob)$, where $\Omega$ is the set of all graphs with vertex set $\{1,2,\dots,n\}$, $\mathcal{F}$ is the family of all subsets of $\Omega$, and for every $G \in \Omega$,
$$
\Prob(G) = p^{|E(G)|} (1-p)^{{n \choose 2} - |E(G)|} \,.
$$
This space may be viewed as the set of outcomes of ${n \choose 2}$ independent coin flips, one for each pair $(u,v)$ of vertices, where the probability of success (that is, adding edge $uv$) is $p.$ Note that $p=p(n)$ may (and usually does) tend to zero as $n$ tends to infinity.  We often write $G(n,p)$ when we mean a graph drawn from the distribution $G(n,p)$.  

\bigskip

Lemma~\ref{lem:gnp_to_gnm} below provides us with a tool to translate results from $G(n,p)$ to $\G(n,m)$---see, for example,~(1.6) in~\cite{JLR}.

\begin{lemma}\label{lem:gnp_to_gnm}
Let $P$ be an arbitrary property, let $m=m(n)$ be any function such that $m \le n \log n$, and take $p=p(n) = m/{n \choose 2}$. Then,
$$
\Prob(\G(n,m) \in P) \le 3 \sqrt{n \log n} \cdot \Prob(G(n,p) \in P).
$$
\end{lemma}

\bigskip

We treat various intervals for $i$ independently, since they require quite different approaches.

\subsection{$i=2$} 

Suppose that the event $B_2$ holds and $I \subseteq W$, $|I|=2$, is such that $\kappa(G_I) \ge n-1$ and both colours from $I$ are present in $G_I$. Any such set $I$ will be called \textbf{bad}. It follows that $G_I$ consists of just one double-edge and no other edges.  

First, we will show that $\G(n,m_-)$ is unlikely to have any such bad set $I$. Indeed, the expected number of bad sets $I$ of size $2$ is equal to
\begin{align*}
\binom{n-1}{2} m_- \frac{\binom{n-3}{k-2}}{ \binom{n-1}{k}} \of{\frac{\binom{n-3}{k}}{\binom{n-1}{k}} }^{m_- -1} &= O\left( m_- \of{\frac{(n-3)_k}{(n-1)_k} }^{m_-} \right)\\
&= O\left( m_- \exp \left( \frac {-2km_-}{n} \left(1 + O(n^{-1}) \right)  \right) \right) \\
&= O\left( m_- n^{-2} \exp(2 \omega) \right) = o(1),
\end{align*}
provided $\omega$ tends to infinity slowly enough. In particular, $\Prob_{m_-}(B_2) = o(1)$ by Markov's inequality. 

Now, assuming that there are no bad sets $I$ of size $2$ in $\G(n,m_-)$, we get that if for some $m > m_-$ there is a bad set $I$ of size 2 in $\G(n,m)$, then $I$ consists of two colours that are both not present in $\G(n,m_-)$. The probability that there are $\o'$ many colours that have not been seen at time $m_-$ is at most 
\begin{align*}
{n-1 \choose \o'} \of{\frac{{n-1 - \o' \choose k}}{{n-1 \choose k}}}^{m_-} &\le \of{ \frac{ne}{\o'}}^{\o'} \exp \left( -\frac{\o' k m_-}{n} \of{1+O\of{\frac{\o'}{n}}} \right)\\
&\le (1+o(1)) \exp \Big( \o' (\log n + 1 - \log \o') - \o' (\log n - \o) \Big)\\
&= \exp \Big( - (1+o(1)) \o' \log \o' + \o' \o \Big),
\end{align*}
for $\o'$ tending to infinity sufficiently slowly. Hence, a.a.s.\ the set of unseen colours in $\G(n, m_-)$ has size at most, say, $\o' = e^{2\omega}$. Finally, we estimate the probability that the $m^{th}$ edge introduced has at least two previously unseen colours to get that 
$$
\sum_{m=m_-}^{m_+} \Prob_m(B_2) \le o(1) + (m_+-m_-) \cdot {\o' \choose 2} \cdot \frac{{n-3 \choose k-2}}{{n-1 \choose k}} = O \left( \frac {\o e^{4\o}}{n} \right) = o(1),
$$
provided $\o$ tends to infinity slowly enough. It follows that the contribution to~(\ref{eq:sum}) from this case is $o(1)$.

\subsection{$i=3$} 

Suppose that the event $C_3$ holds and $I \subseteq W$, $|I|=3$, is such that $\kappa(G_I) \ge n-2$ and no set $I \subseteq W$, $|I|=2$, satisfies $\kappa(G_I) \ge n-1$. As before, any such set $I$ will be called \textbf{bad}. Since $I$ is minimal, $G_I$ contains no cut edges. It follows that $G_I$ must induce isolated vertices and one of the following graphs: a triangle (of either single- or possibly double-edges), a path of length $2$ of double-edges, or two isolated double-edges.

We will bound the expected number of each of these structures in $G(n,p)$ (with $p = m/{n \choose 2}$ for some $m_- \le m \le m_+$), and show that they are all at most $n^{-2+o(1)}$. It will imply, by Markov's inequality, that with probability at least $1-n^{-2+o(1)}$ there is no bad set $I$ in $G(n,p)$ and so the same holds for $\G(n,m)$ with probability at least $1-n^{-3/2+o(1)}$ by Lemma~\ref{lem:gnp_to_gnm}. The contribution to~(\ref{eq:sum}) from this case will be $\sum_{m=m_-}^{m_+} \Prob_m(C_3) = o(1)$.

Fix any $I \subseteq W$ (not necessarily of size 3), any pair of vertices $u,v$, and a number $1 \le x \le k$. Let $\mathcal{E}_x$ ($\mathcal{E}_{\ge x}$) be the event that the edge $uv$ is present in the random graph $G(n,p)$ and has exactly (at least, respectively) $x$ colours from its list in $I$. We have
\begin{eqnarray*}
\Prob \of{\mathcal{E}_x} &&= p \cdot \frac{{i \choose x}{n-1-i \choose k-x}}{{n-1 \choose k}}\\
&&= p \cdot \frac{(i)_x (n-1-i)_{k-x}}{(n-1)_k} \cdot {k \choose x} =  {k \choose x} \cdot \frac{p \cdot (i)_x}{n^x} \of{1+O\of{\frac{i}{n}}},
\end{eqnarray*}
and so
$$
\Prob \of{\mathcal{E}_{\ge x}} =  \Prob \of{\mathcal{E}_x} \of{1+O\of{\frac{i}{n}}} = {k \choose x} \cdot \frac{p \cdot (i)_x}{n^x} \of{1+O\of{\frac{i}{n}}}.
$$

Now, fix $m$ such that $m_- \le m \le m_+$ and let $p = m/{n \choose 2} = (1+o(1)) \log n / n$. The expected number of bad sets $I$ such that the only non-trivial component of $G_I$ is a triangle is at most 
\begin{align*}
&\binom{n-1}{3} \binom{n}{3} \of{\frac{(3+o(1))kp}{n}}^3\of{1-\frac{(3+o(1))kp}{n}}^{\binom{n}{2}-3} \le n^{-3+o(1)}.
\end{align*}
The expected number of bad sets $I$ such that the non-trivial component of $G_I$ is a $2$-path of double-edges is at most 
\begin{align*}
&\binom{n-1}{3} \binom{n}{3} 3 \of{\binom{k}{2}\frac{(6+o(1))p}{n^2}}^2\of{1-\frac{(3+o(1))kp}{n}}^{\binom{n}{2}-2} \le n^{-3+o(1)}.
\end{align*}
Finally, the expected number of bad sets $I$ such that the non-trivial components of $G_I$ are two isolated double-edges is at most 
\begin{align*}
&\binom{n-1}{3} \binom{n}{4} 3 \of{\binom{k}{2}\frac{(6+o(1)) p}{n^2}}^2\of{1-\frac{(3+o(1)) kp}{n}}^{\binom{n}{2}-2} \le n^{-2+o(1)}.
\end{align*}
As before, the contribution to~(\ref{eq:sum}) from this case is $o(1)$.

\subsection{$4 \le i \le n/(\b \log n)$} Suppose that the event $C_i$ holds for some $i$ and $I \subseteq W$, $|I|=i$, is such that $\kappa(G_I) \ge n-i+1$ and no set $I \subseteq W$, $|I|=i-1$, satisfies $\kappa(G_I) \ge n-i+2$. As usual, any such set $I$ will be called \textbf{bad}. 

Suppose the graph induced by $I$ has its nontrivial components on some set of $t$ vertices. Suppose there are $u_1$ single-edges and $u_2$ edges of multiplicity at least 2. For $I$ to be minimal we cannot have any single-edge bridges. So, in particular, no vertex can be incident to just one single-edge and no other edges. Thus the degree of each vertex (counting multiplicity of edges) is at least $2$. It follows, by considering a degree sum where each single edge contributes 1 to an incident vertex and each edge of higher multiplicity contributes 2, that
$$t \le u_1 + 2u_2.$$
Since each non-trivial component has at least two vertices, the number of vertices must satisfy
\begin{eqnarray*}
n &\ge& 2 (\kappa(G_I) - (n-t)) + (n-t) \\
&\ge& 2 ( (n-i+1) - (n-t) ) + (n-t) \\
&=&n - 2(i-1) + t.
\end{eqnarray*}
Hence, we also have 
$$t \le 2(i-1).$$ 
Moreover, 
$$i \le u_1 + ku_2,$$ 
$$u_1 + u_2 \le {t \choose 2}.$$ 

\bigskip

As in the previous case, fix $m$ such that $m_- \le m \le m_+$ and let 
$$
p = \frac {m}{{n \choose 2}} = \frac {\log n + O(\omega)}{n} = (1+o(1)) \frac {\log n}{n}. 
$$
The probability that $C_i$ happens in $G(n,p)$ with given parameters $t, u_1, u_2$ is at most
\begin{align}
&{n-1 \choose i} {n \choose t} {{t \choose 2} \choose u_1 + u_2} {u_1 + u_2 \choose u_2} \of{\frac{(1+o(1))kpi}{n}}^{u_1} \label{eq:pre-f} \\
& \quad \quad \quad \times \of{\of{1 + o(1)} {k \choose 2}\frac{i^2 p}{n^2}}^{u_2} \of{1- \of{1+ O\of{\frac {i}{n}}} \frac{kpi}{n}}^{{n \choose 2} - u_1 - u_2}\nonumber
\end{align}
To see the above expression, first choose $i$ colours, then $t$ vertices for the non-trivial components. The third factor chooses which pairs will have an edge (or double-edge) and the fourth factor chooses which pairs are double-edges. The next two factors account for edge and double-edge probabilities, respectively, and the last factor is the probability that no other edges are present. We upper bound \eqref{eq:pre-f} by
\begin{align*}
& \of{\frac{ne}{i} }^i \of{\frac{ne}{t} }^t \of{\frac{t^2 e}{2(u_1+u_2)} }^{u_1 + u_2} \of{\frac{(u_1+u_2)e}{u_2} }^{u_2} \of{\frac{ekpi}{n}}^{u_1} \\
& \quad \quad \quad \times \of{{k \choose 2}\frac{e i^2 p}{n^2}}^{u_2} e^{-\frac{kpi}{n} \of{{n \choose 2} - u_1 - u_2 } \of{1+ O\of{\frac {i}{n}}} }\\
&=\of{\frac{n}{i} }^i \of{\frac{n}{t} }^t \of{\frac{t }{u_1+u_2} }^{u_1} t^{u_1 + 2u_2}  \of{\frac{pi}{n}}^{u_1} \of{\frac{i^2 p}{u_2 n^2}}^{u_2} e^{-\frac{ik}{2} \log n + O\of{u_1 + u_2 + \omega i + i^2 \log n / n}   }.
\end{align*} 
Since $t/(u_1+u_2)  \le (u_1+2u_2)/(u_1+u_2) \le 2$, the probability in question is at most
\begin{align*}
f(i, t, u_1, u_2) & := n^{t- \of{\frac{k}{2}-1}i} t^{-t} i^{-i}  \of{ \frac{Cit \log n}{n^2}}^{u_1} \of{\frac{Ci^2t^2 \log n }{u_2 n^3}}^{u_2} e^{ C \of{\omega i + i^2 \log n / n} },
\end{align*} 
where $C$ is some universal, sufficiently large, constant. Note that in the current case $i=O(n/\log n)$ so, in fact, $e^{ C \of{\omega i + i^2 \log n / n} } \le e^{ 2 C \omega i}$ but we keep both terms for the future case in which it is only assumed that $i=O(n/\sqrt{\log n})$.  The probability that  $C_i$ happens (with any parameters $t, u_1, u_2$) is at most
$$ \sum_{t=4}^{2(i-1)} \sum_{u_2=0}^{{t \choose 2}} \sum_{u_1 = \max\{0, t-2u_2, i-ku_2\}}^{{t \choose 2}-u_2} f(i,t,u_1, u_2).$$
We will sum the above expression over $i$ ($4 \le i \le n/(\b \log n)$), and bound the sum using three cases according to the value of $\max\{0, t-2u_2, i-2u_2\}$. Let us note that, for convenience, we will treat expressions like $0^0$ to be equal to 1 so that estimations like ${a \choose b} \le (ae/b)^b$ could be applied for all values of $b$, including zero. It is also worth noting that in the sums below some combinations of parameters are not actually possible to occur. However, this convenient approach causes no problem, since each term is positive and we only aim for an upper bound (of $n^{-2+o(1)}$) to be able to apply Lemma~\ref{lem:gnp_to_gnm} and the union bound over all possible values of $m$. Once it is done, the contribution to~(\ref{eq:sum}) from this case is $o(1)$.

\subsubsection{Case 1: $\max\{0, t-2u_2, i-ku_2\} = i-ku_2$} It follows that $u_2 \le i/k$ and $t \le i - (k-2)u_2 \le i$. We want to estimate the following:
\begin{eqnarray*}
\xi_1 &=&\sum_{i=4}^{\frac{n}{\b \log n}} \sum_{t=4}^{i} \sum_{u_2=0}^{\frac{i}{k}} \sum_{u_1 = i-ku_2}^{{t \choose 2}-u_2} f(i,t,u_1, u_2).
\end{eqnarray*}
Note that the innermost sum (over $u_1$) is geometric with ratio $Cit \log n/n^2 = O( i^2 \log n / n^2 ) = o(1)$ and so the inner sum is dominated by its first term. Hence,
\begin{eqnarray*}
\xi_1 &=& O\of{ \sum_{i=4}^{\frac{n}{\b \log n}} \sum_{t=4}^{i} \sum_{u_2=0}^{\frac{i}{k}} n^{t- \of{\frac{k}{2}-1}i} t^{-t} i^{-i}  \of{ \frac{Cit \log n}{n^2}}^{i-ku_2} \of{\frac{Ci^2t^2 \log n }{u_2 n^3}}^{u_2} e^{ 2 C \omega i} } \\
&=&  O\of{ \sum_{i=4}^{\frac{n}{\b \log n}} \sum_{t=4}^{i} \sum_{u_2=0}^{\frac{i}{k}} n^{t- \of{\frac{k}{2}-1}i} t^{i-t}  \of{\frac{C \log n}{n^2}}^i \of{\frac{ n^{2k-3} }{C^{k-1} u_2 i^{k-2}t^{k-2}\log^{k-1} n}}^{u_2} e^{ 2 C \omega i} }.
\end{eqnarray*}
Now, note that the innermost sum (over $u_2$) is of the order of its last term. This follows from the fact that the ratio of consecutive terms is 
\begin{align*}
&\frac{\of{\frac{ n^{2k-3} }{C^{k-1} (u_2+1) i^{k-2}t^{k-2}\log^{k-1} n}}^{u_2+1}}{\of{\frac{ n^{2k-3} }{C^{k-1} u_2 i^{k-2}t^{k-2}\log^{k-1} n}}^{u_2}} = \of{\frac{ n^{2k-3} }{C^{k-1} (u_2+1) i^{k-2}t^{k-2}\log^{k-1} n}} \of{\frac{u_2}{u_2+1}}^{u_2}\\
& =\of{\frac{ n^{2} }{C^{}  i^{}t^{}\log^{} n}}^{k-2} \of{\frac{n}{C(u_2+1)\log n}} \of{\frac{u_2}{u_2+1}}^{u_2} \ge \of{\frac{n}{C(i/2+1)\log n}} e^{-1} \ge 2
\end{align*}
so long as $i$ is at most $n/(\b \log n)$ for $\b > 0$ large enough. We get
\begin{eqnarray*}
\xi_1 &=& O\of{\sum_{i=4}^{\frac{n}{\b \log n}} \sum_{t=4}^{i}  n^{t- \of{\frac{k}{2}-1}i} t^{i-t}  \of{\frac{C \log n}{n^2}}^i \of{\frac{ n^{2k-3} }{C^{k-1} \of{ \frac{i}{k}} i^{k-2}t^{k-2}\log^{k-1} n}}^{\frac{i}{k}} e^{ 2 C \omega i} }\\
&=& O\of{\sum_{i=4}^{\frac{n}{\b \log n}} \sum_{t=4}^{i}  \of{\frac{n}{t}}^{t} \of{\frac{ Ckt^2 \log n}{n^{3+k(\frac{k}{2}-1)} i^{k-1} }}^{\frac{i}{k}} e^{ 2 C \omega i} }.
\end{eqnarray*}
As before, we observe that the innermost sum (over $t$) is of the order of its last term by similar reasoning and looking at the ratio of consecutive terms. It follows that
\begin{eqnarray*}
\xi_1 &=& O\of{\sum_{i=4}^{\frac{n}{\b \log n}}  \of{\frac{n}{i}}^{i} \of{\frac{ Cki^2 \log n}{n^{3+k(\frac{k}{2}-1)} i^{k-1} }}^{\frac{i}{k}} e^{ 2 C \omega i} } \\
&=& O\of{\sum_{i=4}^{\frac{n}{\b \log n}}   \of{\frac{ Ck e^{ 2 C k \omega} \log n }{n^{3+k(\frac{k}{2}-2)} i^{2k-3} }}^{\frac{i}{k}}} =O\of{ \of{\frac{ Ck e^{ 2 C k \omega} \log n}{n^{3+k(\frac{k}{2}-2)}  }}^{\frac{4}{k}}} \\
&\le&  n^{-\frac {12}{k} - 4(\frac{k}{2}-2) + o(1)}  \le n^{-2+o(1)}.
\end{eqnarray*}

\subsubsection{Case 2: $\max\{0, t-2u_2, i-ku_2\}=t-2u_2$} It follows that $u_2 \le t/2$ and we already know that $i \ge t/2+1$. This time, we want to estimate the following:
\begin{eqnarray*}
\xi_2 &=& \sum_{t\ge 4} \sum_{i\ge \frac t2 + 1} \sum_{u_2=0}^{\frac{t}{2}} \sum_{u_1 = t-2u_2}^{{t \choose 2}-u_2} f(i,t,u_1, u_2) \\
&=& O\of{ \sum_{t\ge 4} \sum_{i\ge \frac t2 + 1} \sum_{u_2=0}^{\frac{t}{2}} n^{t- \of{\frac{k}{2}-1}i} t^{-t} i^{-i}  \of{ \frac{Cit \log n}{n^2}}^{t-2u_2} \of{\frac{Ci^2t^2 \log n }{u_2 n^3}}^{u_2} e^{ 2 C \omega i} }.
\end{eqnarray*}
Dropping the term $n^{- \of{\frac{k}{2}-1}i} \le 1$ that is equal to one for $k=2$, we get
\begin{eqnarray*}
\xi_2 &=& O\of{ \sum_{t\ge 4} \sum_{i\ge\frac t2 + 1} \sum_{u_2=0}^{\frac{t}{2}}  i^{-i}\of{\frac{C i\log n}{n}}^t \of{\frac{ n }{C u_2 \log n}}^{u_2} e^{ 2 C \omega i} }\\
&=& O\of{\sum_{t\ge 4} \sum_{i\ge\frac t2 + 1}   i^{-i}\of{\frac{C i\log n}{n}}^t\of{\frac{ n }{C \frac{t}{2} \log n}}^{\frac{t}{2}} e^{ 2 C \omega i} }\\
&=& O\of{\sum_{t\ge 4} \sum_{i\ge\frac t2 + 1}   i^{-i}\of{\frac{ 2Ci^2  \log n}{ tn }}^{\frac{t}{2}} e^{ 2 C \omega i} }.
\end{eqnarray*}
In order to investigate the innermost sum (over $i$), we consider the ratio of consecutive terms of the sequence $i^{-i+t} e^{ 2 C \omega i} $:
$$
\frac{(i+1)^{-i-1+t} e^{ 2 C \omega (i+1)} }{i^{-i+t} e^{ 2 C \omega i} } = \frac{1}{i+1}\of{1 + \frac{1}{i}}^{t-i} e^{ 2 C \omega} .
$$
Since $t\le 2(i-1)$, we have that \[\frac 1e \le\of{1+\frac 1i}^{-i}\le\of{1+\frac{1}{i}}^{t-i}\le \of{1+\frac 1i}^{i}\le e.\]
So the ratio is at most $e^{ 2 C \omega + 1} / (i+1)$. Let $i_0$ be the smallest integer $i$ such that $e^{ 2 C \omega + 1} / (i+1) \le 1/2$. Now, we bound the sum as follows:
\begin{eqnarray*}
\sum_{i\ge\frac t2 + 1} i^{-i+t} e^{ 2 C \omega i} &\le& e^{ 2 C \omega i_0} \sum_{i\ge\frac t2 + 1} i^{-i+t} \max \{ e^{ 2 C \omega (i-i_0)}, 1\} \\
&=& n^{o(1)} \sum_{i\ge\frac t2 + 1} i^{-i+t} \max \{ e^{ 2 C \omega (i-i_0)}, 1\},
\end{eqnarray*}
provided that $\omega$ tends to infinity slowly enough. This time, the ratio of consecutive terms is at most $e^{ 2 C \omega + 1} / (i+1) \le 1/2$ for $i \ge i_0$ and at most $e/(i+1) < 4/5$ for $i < i_0$, since $i \ge t/2+1 \ge 3$ (in fact, $i \ge 4$). It follows that
\begin{eqnarray*}
\xi_2 &\le& n^{o(1)} \sum_{t\ge 4}  \of{\frac t2 +1}^{-\of{\frac t2 +1}}\of{\frac{ 2C\of{\frac t2 +1}^2  \log n}{ tn }}^{\frac{t}{2}} \\
&=& n^{o(1)} \sum_{t\ge 4}  \of{\frac t2 +1}^{-1} \of{\frac{ C\of{t +2} \log n}{ tn }}^{\frac{t}{2}} \le n^{-2+o(1)}.
\end{eqnarray*}

\subsubsection{Case 3: $\max\{0, t-2u_2, i-ku_2\}=0$} It follows that $u_2 \ge t/2$, $u_2 \ge i/k$, and we already know that $t \le 2(i-1)$. In this case, we want to estimate the following: 
\begin{eqnarray*}
\xi_3 &=& \sum_{i=4}^{\frac{n}{\b \log n}} \sum_{t=4}^{2(i-1)} \sum_{u_2=\max \{\frac{t}{2},\frac{i}{k} \}}^{{t \choose 2}} \sum_{u_1 = 0}^{{t \choose 2}-u_2} f(i,t,u_1, u_2) \\
&=& O\of{ \sum_{i=4}^{\frac{n}{\b \log n}} \sum_{t=4}^{2(i-1)} \sum_{u_2=\max \{\frac{t}{2},\frac{i}{k} \}}^{{t \choose 2}} n^t t^{-t} i^{-i}   \of{\frac{Ci^2t^2 \log n }{u_2 n^3}}^{u_2} e^{ 2 C \omega i} } \\
&=& O\of{ \sum_{i=4}^{\frac{n}{\b \log n}} \sum_{t=4}^{2(i-1)} \sum_{u_2=\max \{\frac{t}{2},\frac{i}{k} \}}^{{t \choose 2}} n^t t^{-t} i^{-i}   \of{\frac{kCit^2 \log n }{ n^3}}^{u_2} e^{ 2 C \omega i} } \\
&=& O\of{\sum_{i=4}^{\frac{n}{\b \log n}} \sum_{t=4}^{2(i-1)}  n^t t^{-t} i^{-i}   \of{\frac{kCit^2 \log n }{ n^3}}^{\frac{t}{2}} e^{ 2 C \omega i} }\\
&=& O\of{\sum_{i=4}^{\frac{n}{\b \log n}} \sum_{t=4}^{2(i-1)} i^{-i}  \of{\frac{kCi \log n }{ n}}^{\frac{t}{2}} e^{ 2 C \omega i} } \\
&=& O\of{\sum_{i=4}^{\frac{n}{\b \log n}} i^{2-i} \of{\frac{kC \log n }{ n}}^{2} e^{ 2 C \omega i} }.
\end{eqnarray*}
The ratio of consecutive terms is at most $e^{ 2 C \omega}/(i+1)$ and one can argue as in the previous sub-case that
\begin{eqnarray*}
\xi_3 &\le&  n^{o(1)} \of{\frac{kC \log n }{ n}}^{2} = n^{-2+o(1)}.\\
\end{eqnarray*}

\subsection{$n/(\b \log n) \le i \le n/(\b \sqrt{\log n})$.} In this range of $i$, the terms (in the sum we considered in the previous range) are all exponentially small.  We will estimate the contribution to~(\ref{eq:sum}) from this case using the same notation and strategy. We start with the following estimation that holds for large enough $\beta$:
\begin{eqnarray*}
f(i,t,u_1, u_2)&=&n^{t- \of{\frac{k}{2}-1}i} t^{-t} i^{-i}  \of{ \frac{Cit \log n}{n^2}}^{u_1} \of{\frac{Ci^2t^2 \log n }{u_2 n^3}}^{u_2} e^{ C \of{\omega i + i^2 \log n / n} } \\
&\le& \of{\frac nt}^t i^{-i} \of{ \frac{2C i^2 \log n}{n^2}}^{u_1} \of{\frac{4 Ci^4 \log n }{u_2 n^3}}^{u_2} e^{ C \of{\omega i + i^2 \log n / n} } \\
&\le& \of{\frac nt}^t i^{-i} \of{\frac{4C n }{\b^4 u_2 \log n}}^{u_2} e^{ C \of{\omega i + i^2 \log n / n} } \\
&\le& \of{\frac {n}{2i}}^{2i} \exp \left(-i \log i + O\of{ \frac{n}{\log n}} + O\of{\omega i + i^2 \log n / n} \right) \\
&\le&  \exp \left( 2i(\log \log n + \log \b) - (1+o(1)) i \log n + O\of{ \frac{n}{\log n} } \right)\\
&\le& \exp \Big(- \Omega\of{n} \Big),
\end{eqnarray*}
provided $\omega$ tends to infinity slowly enough. It follows that 
\begin{eqnarray*}
\xi_4 &=&\sum_{i=\frac{n}{\b \log n}}^{\frac{n}{\b \sqrt{\log n}}} \sum_{t} \sum_{u_2} \sum_{u_1} f(i,t,u_1, u_2) \le n^{-2+o(1)}.
\end{eqnarray*}

\subsection{$n/(\b\sqrt{\log n}) \le i \le n - \b n/\log n$}

From now on, we start thinking of multigraph as a graph (that is, we stop caring about edge multiplicity but only whether a given edge of $\G(n,m)$ occurs in $G_I$ or not). In this section, we will be using the following concentration inequalities. Let $X \in \Bin(n,p)$ be a random variable with the binomial distribution with parameters $n$ and $p$. Then, a consequence of Chernoff's bound (see e.g.~\cite[Corollary~2.3]{JLR}) is that 
\begin{equation}\label{chern}
\Prob( |X-\E X| \ge \eps \E X) \le 2\exp \left( - \frac {\eps^2 \E X}{3} \right)  
\end{equation}
for  $0 < \eps < 3/2$. We will also apply the bound of Bernstein (see e.g.~\cite[Theorem~2.1]{JLR}) that for every $x > 0$, 
\begin{equation}\label{Bernstein1}
\Prob \left( X \ge (1+x) \E X \right) \le \exp \left( - \E X \varphi(x) \right) \le \exp \left( - \frac {x^2  \E X} {2(1+x/3)} \right),
\end{equation}
and for every $0 < x < 1$,
\begin{equation}\label{Bernstein2}
\Prob \left( X \le (1-x) \E X \right) \le \exp \left( - \E X \varphi(-x) \right) \le \exp \left( - \frac {x^2  \E X} {2} \right),
\end{equation}
where $\varphi(x)=(1+x)\log (1+x) - x$.

\bigskip

We start with investigating some typical properties of $\G(n,m_-)$. For convenience, let $p_- = m_- / {n \choose 2} = (\log n - \omega+o(1))/n$ and consider $G(n,p_-)$ instead of $\G(n,m_-)$. Let $L$ be the set of vertices with $\deg(v) < \frac{1}{10} \log n$. For a given vertex $v$ in $G(n,p)$, we have $\E \deg(v) = p_- (n-1) = \log n - \omega + o(1)$. It follows from~(\ref{Bernstein2}) that
\begin{eqnarray*}
\Prob \of{ \deg(v) < 0.1 \log n} &\le& \Prob \of{ \deg(v) \le 0.11 \E \deg(v)} \\
&\le& \exp\of{- \E \deg(v) \cdot\of{0.11 \log(0.11) + 0.89}} \le n^{-0.6}.
\end{eqnarray*}
Hence, the expected size of $L$ is at most $n^{0.4}$ and it follows from Markov's inequality that with probability at least $1-n^{-0.55}$ it is smaller than $n^{0.95}$. By Lemma~\ref{lem:gnp_to_gnm}, the same property holds a.a.s.\ for $\G(n,m_-)$ and so we may condition on the fact that $\G(n,m_-)$ satisfies $|L| < n^{0.95}$. 

Now, we are going to orient the edges of $\G(n,m_-)$, which will turn out to be a convenient way to avoid events being dependent. With the goal of showing the existence of such orientation, suppose we randomly orient the edges of $\G(n,m_-)$ and let $\deg^+(v)$ represent the out-degree of vertex $v$. Call a vertex \textbf{very bad} if $\deg^+(v) \le \frac{1}{40} \log n$. Then for any vertex $v \notin L$ (that is, with degree at least $\frac{1}{10}\log n$), we have that $\deg^+(v)$ is stochastically dominated (from below) by $\Bin\of{\frac{1}{10}\log n, \frac{1}{2}}$.
Thus, by~(\ref{chern}), we have
\[
\Prob \of{\deg^+(v) \le \frac{1}{40}\log n} \le \Prob \of{   \Bin\of{\frac{1}{10}\log n, \frac {1}{2}} \le \frac{1}{40}\log n} \le 2 \exp\of{-\frac{1}{12\cdot 20} \log n}.
\]
So the expected number of very bad vertices is at most $|L| + 2 n^{1-1/240} \le n^{0.999}$. Thus, by the basic probabilistic method, there exists an orientation with at most this many very bad vertices. (Recall that $\G(n,m_-)$ is now treated as any deterministic graph with $|L| < n^{0.95}$.)

\bigskip

Our goal is to show that the event $A_i$ (in the range for $i$ considered in this case) does not hold in $\G(n,m)$ with $m_- \le m \le m_+$. However, if $A_i$ does not hold in $\G(n,m_-)$, then it cannot hold in $\G(n,m)$ with $m > m_-$, since,  for a given $I \subseteq W$, adding edges can only decrease $\kappa(G_I)$, the number of components in $G_I$. Hence, it remains to focus on $\G(n,m_-)$. 

We condition on $\G(n,m_-)$ having orientation with at most $n^{0.999}$ very bad vertices. Note that, since edges in the random graph process and colours are generated independently, we may start with $\G(n,m_-)$ (and its orientation) and then test all sets of colours. Hence, fix a set of colours $I \subseteq W$, with $|I|=i$ and $n/(\b \sqrt{\log n}) \le i \le n - \b n/\log n$. Let $\deg_I^+(v)$ represent the out-degree of $v$ in $G_I$,  where $G_I$ is the (oriented) subgraph of $\G(n,m_-)$ consisting of all edges which have at least one colour from $I$ on their list and the orientations are retained from $\G(n,m_-)$. The probability that a given edge $e \in \G(n,m_-)$ is in $G_I$ is equal to
$$
1 - \frac { {n-1-i \choose k} }{ {n-1 \choose k} } = 1 - \of{ 1- \frac {i}{k} }^k \of{ 1 + O\of {\frac {1}{n-i} } } = 1 - \of{ 1- \frac {i}{k} }^k + O\of {\frac {1}{n} } \ge \frac {i}{n}.
$$
In particular, note that for $i$ in this range, the average degree in $G_I$ is at least $\frac {2m_-}{n} \cdot \frac {i}{n} = \Theta(i \log n / n) = \Omega(\sqrt{\log n}) \to \infty$ as $n\to \infty$. For a vertex $v$ which is not very bad, we have that $\deg_I^+(v)$ is binomial and stochastically dominates $\Bin\of{\frac{1}{40} \log n, \frac{i}{n}}$. We call a vertex $v$ in $G_I$ \textbf{bad} if $\deg^+_I(v)$ satisfies
\[
\deg_I^+(v) \le \frac 12 \cdot\frac 1{40}\log n \cdot\frac{i}{n} = \frac{1}{80} \cdot \frac {i \log n}{n}.
\]
So, it follows from~(\ref{chern}) that the probability that a non-very bad vertex is bad is at most
\begin{align*}
\Prob \of{ \Bin\of{\frac{1}{40} \log n, \frac{i}{n}} \le \frac{1}{80} \cdot \frac {i \log n}{n}} &\le \exp\of{ \frac{1}{480} \cdot \frac {i \log n}{n} }.
\end{align*}

Now, let $X_b=X_b(I)$ represent the number of bad vertices in $G_I$ that are not very bad. The expectation of $X_b$ is at most 
\[
E_b=E_b(I):= n\cdot \exp\of{- \frac{1}{480} \cdot \frac {i \log n}{n} }.
\] 
Recall that the reason for introducing the orientation of $\G(n,m_-)$ is to make the events of being bad to be independent of one another. Hence, $X_b$ is stochastically dominated by random variable $Z \sim \Bin\of{n, \exp\of{-i\log n/(480n)}}$ with expectation $E_b$. Thus, using~(\ref{Bernstein1}) we get that
\begin{align*}
\Prob \of{X_b \ge E_b + \gamma n} &\le \Prob \of{Z \ge E_b + \gamma n}\\
&\le \exp\of{-E_b\cdot\of{\of{1+\frac{\gamma n}{E_b}}\log\of{1+\frac{\gamma n}{E_b}} - \frac{\gamma n}{E_b}}}.
\end{align*}
Let us set $\gamma = \frac{500 n}{i \log n} \to 0$ as $n \to \infty$. Then, 
$$
\gamma n /E_b = \gamma \exp \of{ \frac {i\log n}{480 n} } = \Omega(\log^{-1} n) \cdot \exp \of{ \Omega( \sqrt{\log n} ) } \to \infty
$$ 
as $n \to \infty$. Thus, we get
\begin{eqnarray*}
\Prob \of{X_b \ge E_b + \gamma n} &\le& \exp \of{ -\gamma n \log \of{ \frac {\gamma n}{E_b} } (1+o(1)) }\\
&=& \exp\of{- \frac{500 n^2}{i \log n}\of{ \log \gamma  + \frac{i \log n}{480 n}}(1+o(1)}\le \exp\of{-n},
\end{eqnarray*}
since $\log \gamma = O (\log \log n) = o (\sqrt{\log n}) = o( i \log n / n)$.

Finally, we may union bound over all choices of $I$ in this range to get that 
$$
\Prob \of{\exists I\,:\,\frac{n}{\beta\sqrt{\log n}} \le i \le n - \frac{\b n}{\log n}\,,\,X_b(I) \ge E_b(I) + \gamma n}
\le 2^n\cdot \exp\of{-n} = o(1).
$$
Thus, a.a.s.\ every $G_I$ has at most $n^{0.999} + (1+o(1) \gamma n = (1+o(1)) 500 n^2 / (i \log n)$ many bad vertices. Let us condition on this. Every vertex which is not bad is in a component of size at least $(i \log n) / (80 n)$. Hence we have
\[\kappa(G_I)\le (1+o(1))  \frac {500 n^2}{i \log n} + \frac{80n^2}{i \log n} \le \frac{600n^2}{i \log n} \]
and this is less than $n- i$ as long as $i \le n - \b n/\log n$ for $\b$ sufficiently large. The contribution to~(\ref{eq:sum}) from this case is $o(1)$.

\subsection{ $n-\beta n/\log n \le i \le n-1-k$}

In this range for $i$, we find it convenient to think of the complement of the set $I$, that is, $J:=W\sm I$. Thus Edmond's condition~(\ref{eq:condition}) becomes 
\[\kappa(G_{W\sm J}) \le |J|+1.\] 
As usual, any set $J$ which does not satisfy this condition will be called \textbf{bad}. We will show that with probability at least $1-n^{-1.8+o(1)}$ there is no bad set $J$ of cardinality at least $k$ and at most $\beta n/\log n$ in $G(n,p)$ (with $p = m/{n \choose 2}$ for some $m_- \le m \le m_+$) and so the same holds for $\G(n,m)$ with probability at least $1-n^{-1.3+o(1)}$ by Lemma~\ref{lem:gnp_to_gnm}. The contribution to~(\ref{eq:sum}) from this case will be $\sum_{m=m_-}^{m_+} \sum_{i=n-\b n / \log n}^{n-1-k} \Prob_m(D_i) = o(1)$.

\bigskip

Given a bad set $J$ with $|J|=j$ such that $k \le j \le \beta n / \log n$, $G_{W\sm J}$ has at least $j+2$ components, and so some collection of these components contains $s$ vertices in total, where $j+1 \le s \le n/2$. Indeed, suppose that $G_{W\sm J}$ has $\xi \ge j+2$ components of orders $n_1 \le n_2 \le \ldots \le n_{\xi}$. If $n_{\xi} \ge n/2$, then $n_1 + n_2 + \ldots +n_{\xi-1}$ is at most $n/2$ and clearly at least $\xi-1 \ge j+1$. If $j+1 \le n_{\xi} < n/2$, then taking the largest component suffices. Finally, if $n_{\xi} \le j$, then we can choose $r$ such that $n_1 + n_2 + \ldots +n_{r} \le n/2$ but $n_1 + n_2 + \ldots +n_{r+1} > n/2$, and then
$$
n_1 + n_2 + \ldots +n_{r} > n/2 - n_{r+1} \ge n/2 - j \ge j+1,
$$
for $n$ sufficiently large, since $j = o(n)$.

First, we will show that a.a.s.\ there is no set of colours $J$ of size $j$ (in the desired range) and no set of vertices $S$ of size $s$ such that 
$$\jlo \le s\le n/2$$ 
(in particular, $s \to \infty$ as $n \to \infty$) with no edges between $S$ and $V\sm S$ in $G_{W\sm J}$.  Indeed, the expected number of such pairs of sets $J$ and $S$ is at most 
$$
X = \sum_{j=k}^{\beta n / \log n} \;\;\; \sum_{s=\jlo}^{ n/2 }\binom{n-1}{j} \binom{n}{s} (1-p_I)^{s(n-s)},
$$
where $p_I$ is an edge probability in $G_I = G_{W\sm J}$. Note that
\[p_I \ge \frac {m_-}{ {n \choose 2} } = \frac{\log n - \omega}{n}\cdot \of{1-\frac{\binom{j}{k}}{\binom{n-1}{k}}} = \frac{\log n - \omega(1+o(1))}{n}.\]
Hence,
\begin{align*} 
X & \le \sum_{j=k}^{\beta n / \log n} \;\;\;\sum_{s=\jlo}^{n/2} \exp\left\{ j \log \of{\frac{ne}{j}}+ s \log \of{\frac{ne}{s}}- s(n-s)p_I\right\}\\
& \le \sum_{j=k}^{\beta n / \log n} \;\;\;\sum_{s=\jlo}^{n/2} \exp\left\{ j \log \of{\frac{ne}{j}} - s \log s +s^2 \frac{\log n}{n}(1+o(1)) +s\omega(1+o(1)) \right\}.
\end{align*}
Note that the inner sum is dominated by its first term. To see this let $a_s$ be the $s$th term and consider the ratio of consecutive terms:
\begin{align*}
\frac{a_{s+1}}{a_s} &= \exp\set{s\log s - (s+1)\log(s+1) +(2s+1) \frac{\log n}{n}(1+o(1)) + \omega(1+o(1))}\\
&=\exp\set{-\log(s+1) - s\log\of{\frac{s+1}{s}} + \of{\frac{2s\log n}{ n} + \omega}\opoo}\\
&=\exp\set{-\log(s+1) + \of{\frac{2s\log n}{ n} + \omega}\opoo}.
\end{align*} 
Suppose $\jlo \le s\le n/\log n$. Then $s=\Omega(\log n)$ and $s\log n /n=O(1)$ so we have
\[\frac{a_{s+1}}{a_s}=\exp\set{-\log(s+1)+ O( \omega)} = o(1)\]
as long as $\omega$ grows sufficiently slowly.
Now suppose $n/\log n \le s \le 0.49 n$. Then
\[\frac{a_{s+1}}{a_s} \le \exp\set{-\log n + 2\frac{0.49n\log n}{n} + O\of{\omega + \log\log n}} = o(1).\]
Thus $\sum_{s=\jlo}^{ 0.49 n } a_s$ is of order of its first term which is bounded by
\begin{align*}
&\quad\exp\set{j\log\of{\frac nj} +j -\jlo \of{\log\of{\jlo} - \omega\opoo} }\\
&\le\exp\set{-j\log\of{\frac nj} +j - \jlo \cdot\of{\log 2 + \log\log \of {n/j} - \log\log j - \omega\opoo} }\\
&\le\exp\set{-j\log\of{\frac nj} \of{ 1 - o(1) + \frac {2}{\log j} \cdot \of{ \log 2 + \log \log (n/j) -  \log \log j - \omega\opoo} }}\\
&\le \exp\set{-.9 j\log\of{\frac nj}},
\end{align*}
since for $j\le \log n$, we have
\[ \log 2 + \log\log \of {n/j} - \log\log j - \omega\opoo   = \log\log (n/j) \opoo \to \infty \]
and for $j > \log n$, we have
\[\frac{- \log\log j - \omega}{\log j} = - o(1),\]
where in both cases we assume $\omega\to \infty$ sufficiently slowly. Moreover,
\begin{align*}
\sum_{s=.49n}^{ n/2 }\binom{n-1}{j} \binom{n}{s} (1-p_I)^{s(n-s)} &\le \sum_{s=.49n}^{ n/2 }2^n\cdot 2^n\cdot \exp\of{- \Omega(n^2p_I)}
&\le \exp\of{-\Omega(n\log n)},
\end{align*}
and so 
$$ 
X = O \of{ \sum_{j=k}^{\beta n / \log n} \exp\left\{-0.9 j \log \of{\frac{n}{j}}  \right\} }.
$$ 
It remains to show that it is of order at most $n^{-1.8}$. To see this, note that the above sum is dominated by its first term since the ratio of consecutive terms is 
\[\frac{\of{\frac{j+1}{n}}^{0.9(j+1)}}{\of{\frac{j}{n}}^{0.9j}} = {\bfrac{j+1}{n}}^{0.9} \of{ \frac{j+1}{j}}^{0.9j} = O \of{ {\bfrac{j+1}{n}}^{0.9} } = o(1). \]

\bigskip

We now consider the case when $s \le \jlo$. When  $n^{3/4}\le j\le \beta n/\log n$, we have that
\[ s \le \jlo \le 2j\cdot \frac{\frac 14 \log n}{\frac 34 \log n} = \frac 23 j < j+1,\]
which is a contradiction.
Thus the previous case covers all $s$ for such $j$ and so in this case we only consider $j\le n^{3/4}$. Our technique for this range of $s$ is slightly different. In the previous case, $s$ was large enough that in $G_{W\sm J}$, no set of $s$ vertices could have no edges to its complement. Now, we allow for such a possibility. Since for a bad set $J$ we have $\kappa(G_{W\sm J}) \ge j+2$, we can partition the vertices into sets of size $x_1,x_2,\ldots, x_{j+1}$ and $n-s$ (where $s = x_1 +\cdots + x_{j+1}$) such that there are no edges crossing these parts in $G_{W\sm J}$.  Since $G=G_W$ is connected, there must be $j+1$ edges in $G$ whose colour sets are subsets of $J$ (and therefore they will not occur in $G_{W\sm J}$) where each edge has at least one endpoint in $S$.   Below, we bound the expected number of pairs $J$ and $S$ satisfying the above description.
$$
Y = \sum_{j = k}^{n^{3/4}}\;\;\ \sum_{s=j+1}^{ \jlo } \;\;\ \sum_{x_1 + \ldots + x_{j+1} = s} \binom{n-1}{j} \frac{(n)_s}{\prod_{\i} x_\i !} \cdot\binom{sn}{j+1} \of{1-p_I}^{s\of{n-s}-j-1 } \cdot  \of{p \cdot  \frac{\binom{j}{k}}{\binom{n-1}{k}}}^{j+1},
$$
where $p_I$ is defined as before and $p \le m_+/{n \choose 2} = (\log n + \omega)/n$. To understand the terms in the above sum, first choose $j$ colours, then partition the vertices into sets of sizes $x_1,\ldots, x_j, n-s$ ($\frac{(n)_s}{\prod_{\i} x_\i !}$ is the appropriate multinomial coefficient for this task). The next factor is an upper bound on the number of choices for $j+1$ edges with one endpoint in $S$. There are at least $s(n-s) - (j +1)$ non-edges in $G_{W\sm J} = G_I$ which explains the next factor. The last one is the probability that the specified $j+1$ edges appear in $G_W$ but not $G_{W\sm J}$.
We get
\begin{align*}
Y &\le \sum_{j = k}^{n^{3/4}}\;\;\ \sum_{s=j+1}^{\jlo} \;\;\ \sum_{x_1 + \ldots + x_{j+1} = s} \exp \Big\{ j \log \of{\frac{ne}{j}} + s \log n - \sum x_\i \log \of{ \frac{x_\i}{e}}+ (j+1) \log \of{\frac{sne}{j+1}} \\
& - p_I \cdot [s\of{n-s}-j-1]   + (j+1)k \log j - (j+1)(k+1)\log n +(j+1) \log (\log n + \omega) \Big\}.
\end{align*}
We use convexity of the function $x\log(x/e)$ and Jensen's inequality to bound the sum of such terms. Keeping in mind that $j+1 \le s = O(j\log n)$ and $j \le n^{3/4}$ we also have that
\begin{align*}
-p_I\of{s(n-s)-j-1} &\le - \bfrac{\log n - \omega}{n}\of{1-\frac{\binom{j}{k}}{\binom{n}{k}}}\of{sn - (s^2 + j + 1)}\\
&= - \bfrac{\log n - \omega}{n}\of{1-O\of{n^{-1/2}}}\cdot sn\of{1+O\bfrac{s}{n}}\\
&=-s(\log n - \omega)\of{1+O\bfrac{\log n}{n^{1/4}}}\\
&=- s\log n + s\omega\opoo.
\end{align*}
Hence,
\begin{align*}
Y &\le \sum_{j = k}^{n^{3/4}}\;\;\ \sum_{s=j+1}^{\jlo} \;\;\ \sum_{x_1 + \ldots + x_{j+1} = s} \exp \Big\{ j \log \of{\frac{ne}{j}}  - (j+1)\cdot \frac{s}{j+1} \log \of{\frac{s}{(j+1)e}} \\
& + s\omega(1+o(1))  + (j+1) \log \of{\frac{sne}{j+1}} + (j+1)k \log j - (j+1)(k+1)\log n +(j+1) \log \log n \Big\}.
\end{align*}
The number of terms in the inner-most sum is bounded from above by $\binom{s}{j} \le \of{\frac {se}{j}}^j$ and by expanding $\log$s and collecting like terms, we get
\begin{align*}
Y &\le \sum_{j = k}^{n^{3/4}}\;\;\ \sum_{s=j+1}^{\jlo}  \exp\Big\{j \log \of{\frac{se}{j}} -\sqbs{(k-1)j +k } \log n + \sqbs{j+1-s} \log s \\
&\qquad + \sqbs{(k-1)j+k} \log j + \of{s-j-1} \log (j+1) +(j+1) \log \log n + s\omega\opoo \Big\} \\
&\le \sum_{j = k}^{n^{3/4}}\;\;\ \sum_{s=j+1}^{\jlo}  \exp \Big\{-\sqbs{(k-1)j +k } \log \of{\frac{n}{j}} - \of{s-j-1} \log \of{\frac {s}{j+1}} \\
&\qquad +s\omega \opoo  +O\of{j \log \log n}  \Big\},
\end{align*}
where we have included $j\log\bfrac{se}{j}$ in the $O(j\log\log n)$ term since $s/j = O(\log n).$  Hence,
\begin{align}
Y &\le \sum_{j=k}^{n^{3/4}} \exp \left\{ -\sqbs{(k-1)j +k}\log \bfrac{n}{j} + O(j\log\log n)\right\} \label{eq:dubsum}\\
&\qquad\quad \times \of{\sum_{s=j+1}^{\jlo}\exp\set{-(s-j-1)\log\bfrac{s}{j+1} + s\omega \opoo}}. \nonumber
\end{align}
We will now split the inner sum at $s_0=j\log\log n/\omega$.
Then for $s\le s_0$ we have
\begin{align*}
\sum_{s=j+1}^{s_0} & \exp\set{-(s-j-1)\log\bfrac{s}{j+1} + s\omega \opoo}\\
&\le s_0\exp\set{s_0\omega \opoo} \\
&= \exp\set{j\log\log n\opoo},
\end{align*}
provided that, say, $\omega = o(\log \log n)$.
For $s\ge s_0$ we have
\begin{align*}
\sum_{s=s_0}^{\jlo} & \exp\set{-(s-j-1)\log\bfrac{s}{j+1} + s\omega \opoo}\\
&\le \sum_{s=s_0}^{\jlo}\exp\set{-s\log\bfrac{s_0}{j+1}\opoo + s\omega \opoo}\\
&\le \sum_{s=s_0}^{\jlo}\exp\set{-s\log\bfrac{s_0}{j+1}\opoo}
\end{align*}
as long as $\omega_0 = o(\log\log\log n)$. This sum is dominated by its first term, that is, 
\[\exp\set{-s_0 \log\bfrac{s_0}{j+1}\opoo} = o(1).\]
Therefore the inner sum in \eqref{eq:dubsum} can be bounded by $\exp\set{j\log\log n\opoo}$.
As a result
$$
Y = O \of{ \sum_{j=k}^{n^{3/4}}\left(-\of{(k-1)j +k}\log \bfrac{n}{j} + O(j\log\log n)\right) }
$$
which is dominated by its first term as seen by examining the ratio of consecutive terms. Thus, $Y$ is at most $n^{-4 + o(1)}$ and so the probability  there is no bad set $J$ of cardinality at least $k$ and at most $\beta n/\log n$ in $G(n,p)$ is at least $1-n^{-1.8+o(1)}$ and the final case to consider is finished.

\end{document}